\documentclass[12pt]{amsart}

\usepackage{amssymb}

\usepackage{enumerate}

\usepackage{graphicx}

\makeatletter
\@namedef{subjclassname@2010}{%
  \textup{2010} Mathematics Subject Classification}
\makeatother

\newtheorem{thm}{Theorem}[section]

\newtheorem{lem}{Lemma}

\theoremstyle{definition}

\newtheorem{rem}[thm]{Remark}

\numberwithin{equation}{section}

\begin{document}


\baselineskip=17pt



\title[positive-rank elliptic curves arising from Pythagorean triples ]{ positive-rank elliptic curves arising from Pythagorean triples  }

\author[F. Izadi]{Farzali Izadi}
\address{Farzali Izadi \\ Department of Mathematics \\ Faculty of Science \\ Urmia University \\ Urmia 165-57153, Iran}
\email{f.izadi@urmia.ac.ir}

\author[M. Baghalaghdam]{Mehdi Baghalaghdam}
\address{Mehdi Baghalaghdam\\ Department of  Mathematics \\ Faculty of Science\\  Azarbaijan Shahid Madani University\\Tabriz 53751-71379, Iran}
\email{mehdi.baghalaghdam@yahoo.com}

\date{}

\begin{abstract}
In the present  paper, we introduce some new families of elliptic curves with positive rank arrising from Pythagorean triples. We study  elliptic curves of the form $y^2=x^3-A^2x+B^2$, where $A, B\in\{a,b,c \}$ are two different  numbers and $(a,b,c)$ is a  Pythagorean triple ($a, b, c\in \mathbb{Q}$). First we prove that if $(a,b,c)$ is a primitive Pythagorean triple(PPT), then the rank of each family is positive. Then we  constract subfamilies
 of rank at least $3$ in each family but one with rank at least two, and obtain elliptic curves of high rank in each family. Furthermore, we consider two other new families of elliptic curves of the forms $y^2=x(x-a^2)(x+c^2)$, and $y^2=x(x-b^2)(x+c^2)$, and prove that if $(a,b,c)$ is a PPT, then the rank of each family is positive.
\end{abstract}



\maketitle

\section{introduction}
 \noindent An elliptic curve (EC) over the rationals is a curve $E$ of genus $1$, defined over $\mathbb{Q}$,
together with a $\mathbb{Q}$-rational point, and is expressed by the generalized Weierstrass equation of the form $$E
: y^2+a_1xy+a_3y=x^3+a_2x^2+a_4x+a_6,$$ where
$a_1,a_2,a_3,a_4,a_6\in \mathbb{Q}$.\\

\noindent A theorem of Mordell-Weil \cite{11} states that
the rational points on $E$, form a finitely generated abelian group $E(\mathbb{Q})$ under
a natural group law, i.e., $E(\mathbb{Q})\cong\mathbb{Z}^r \times
E(\mathbb{Q})_\text{tors} $, where $r$ is a nonnegative integer called the rank of E, and
$E(\mathbb{Q})_\text{tors}$ is the subgroup of elements of finite order
in $E(\mathbb{Q})$, called the torsion subgroup of
$E(\mathbb{Q})$.
 The rank of $E$ is the rank of the free part of this
group.\\

\noindent By Mazur's theorem \cite{9}, the torsion subgroup $E(\mathbb{Q})_\text{tors}$ is
one of the following $15$ groups: $\mathbb{Z}/n\mathbb{Z}$ with
$1\leq n\leq 10$ or $n=12$, $\mathbb{Z}/2\mathbb{Z}\times
\mathbb{Z}/2m\mathbb{Z}$ with $1\leq m\leq 4$.\\

\noindent  Currently there is no general unconditional algorithm to compute
the rank. It is not
known which integers can occur as ranks, but a well-know
conjecture says that the rank can be arbitrarily large. Elliptic curves of large rank are hard to find and the current record
is a curve of rank at least $28$, found by Elkies in $2006$. (see\cite{1})\\

\noindent Also a recent paper by J. Park et al. \cite{7} presents a heuristic suggesting that there are only finitely many elliptic curves of rank greater than $21$. Their heuristic based on modeling the ranks and Shafarevich-Tate groups of elliptic curves simultaneously, and relies on a theorem counting alternating integer matrices of specified rank. Also in a paper by B. Naske,cki \cite{6}, proved that for a generic triple the lower bound of the rank of the EC over $\mathbb{Q}$ is $1$, and for some explicitly given infinite family the rank is $2$. To each family, the author attach an elliptic surface fibred over the projective line and show that the lower bounds for the rank
are optimal, in the sense that for each generic fiber of such an elliptic surface its corresponding Mordell-Weil group over the function field $\mathbb{Q}(T)$ has rank $1$ or $2$ respectively.
\\

\noindent \emph{Specialization} is a significant technique  for finding a lower bound of the rank of a family of elliptic curves. One can consider an EC on the rational function field  $\mathbb{Q}(T)$ and then obtain  elliptic curves over $\mathbb{Q}$ by specializing the variable $T$ to suitable values $t\in\mathbb{Q}$ (see \cite[Chapter III,Theorem 11.4]{10} for more information).\\
 Using this technique, Nagao and Kauyo \cite{5} have found curves of rank$\geq21$, and  Fermigier \cite{2} obtained a curve of rank $\geq22$.\\

\noindent In order to determine $r$, one should find the generators of the free part of the Mordell-Weil group. Determining the \emph{ associated height matrix} is a useful technique for finding a set of generators.

\noindent If  the determinnat of associated height matrix is nonzero, then  the given points are linearly
independent and  $ \text{rank}(E(\mathbb{Q}))\geq r $ (see\cite[ChapterIII]{10} for
more information).
\\

\noindent In this paper,  we study  elliptic curves of the form $y^2=x^3-A^2x+B^2$, where $A, B\in\{a,b,c \}$ are two different  numbers and $(a,b,c)$ is a  Pythagorean triple ($a, b, c\in \mathbb{Q}$). First we prove that if $(a,b,c)$ is a primitive Pythagorean triple(PPT), then the rank of each family is positive. By using both \emph{specialization} and \emph{associated height matrix} techniques, we  constract subfamilies
 of rank at least $3$ in each family but one with rank at least two, and obtain elliptic curves of high rank in each family. Furthermore, we consider two other families of elliptic curves of the forms $y^2=x(x-a^2)(x+c^2)$, and $y^2=x(x-b^2)(x+c^2)$, and prove that if $(a,b,c)$ is a PPT, then the rank of each family is positive. These familes are similar to another family of curves $y^ 2 = x(x-a ^2)(x+b^ 2)$ with $a^ 2 +b^ 2 = c^ 2$ which is a special case of the well-known Frey family.
 In \cite{3}, a subfamily of the elliptic curve  $y^2=x^3-c^2x+a^2$, with the rank at least $4$, has been introduced. In \cite{4}, it is proven that the rank of the elliptic curve $y^2=x(x-a^2)(x-b^2)$, is positive and also in \cite{6} a subfamily of this elliptic curve with the rank at least $2$ is obtained.
 
\noindent  We need two standard facts in this paper:
  \begin{lem}  
The following  relations will generate all primitive Pythagorean triples ($a^2+b^2=c^2$, $(a,b,c)=1$):
$a=m^2-n^2$, $b=2mn$, $c=m^2+n^2$,\\
where $m$, and $n$, are positive integers with $m > n$, and with $m$ and $n$ coprime and not both odd.
  \end{lem}
  
  \begin{lem} (Nagell-Lutz theorem) Let  $y^2=f(x)= x ^3 + ax^2 + bx + c$, be a non-singular cubic curve with integer coefficients  $a, b, c \in \mathbb{Z}$, and let $D$ be the discriminant of the cubic polynomial $f(x)$,\\
$D=-4a^3c+a^2b^2+18abc-4b^3-27c^2$. Let $P = (x, y) \in E(\mathbb{Q})$ be a rational point of finite order. Then  $x$ and  $y$ are integers  and, either $y = 0$, in which case P has order two, or else $y$ divides $D$. (see \cite{9}, page: 56 )
\end{lem}
\section{The EC $y^2=x^3-a^2x+c^2$}
\noindent In each family, let first $(a,b,c)$ be a PPT.  
\noindent First by letting $-a^2x+c^2=0$, in the above elliptic curve, we get $x=\frac{c^2}{a^2}$, and $y=\frac{c^3}{a^3}$. Then the point $(\frac{c^2}{a^2},\frac{c^3}{a^3})$ is on the aforementioned elliptic curve. Note that this point is of infinite order, because in a PPT we have $(a,c)=1$ and $c\neq1$, i.e., the numbers $\frac{c^2}{a^2}$ and $\frac{c^3}{a^3}$ are not integers, then by lemma $2$, the rank of the above elliptic curve is positive.

\noindent Second  we look at
\begin{equation}\label{71}
E: y^2=x^3-a^2x+c^2,
\end{equation}
 as a 1-parameter family by letting
\begin{equation}
a=t^2-1,\quad b=2t,\quad c=t^2+1,
\end{equation}
where $t\in\mathbb{Q}$. Then instead of \eqref{71} one can take
\begin{equation}\label{E292}
E_t: y^2=x^3-(t^2-1)^2x+(t^2+1)^2, \quad t\in\mathbb{Q}.
\end{equation}

\begin{thm}
There are infinitely many elliptic curves of  the  form \eqref{E292}  with rank $\geq3$.
\end{thm}

\begin{proof}
Clearly we have two points
\begin{equation}\label{E23}
P_t=(0,\  t^2+1), \quad Q_t=(t^2-1,\  t^2+1).
\end{equation}
\noindent Now we impose a point on \eqref{E292} with $x$-coordinate equal to $1$. It implies that
$1+4t^2$,
to be a square, say $= v^2$. Hence
\begin{equation}
t=\frac{\alpha^2-1}{4\alpha},\quad v=\frac{\alpha^2+1}{2\alpha},
\end{equation}
with $\alpha\in\mathbb{Q}$.
 Hence instead of \eqref{E292}, one can take
 \begin{equation}\label{E250}
 E_\alpha: y^2=x^3-\left(\left(\frac{\alpha^2-1}{4\alpha}\right)^2-1\right)^2x+\left(\left(\frac{\alpha^2-1}{4\alpha}\right)^2+1\right)^2,
 \end{equation}
 or
 \begin{equation}\label{293}
   E_\alpha:y^2=x^3-(\frac{\alpha^4-18\alpha^2+1}{16\alpha^2})^2x+(\frac{\alpha^4+14\alpha^2+1}{16\alpha^2})^2
 \end{equation}
equipped with the three points
\begin{equation*}
\begin{array}{l}
P_\alpha=\left(0,\left(\frac{\alpha^2-1}{4\alpha}\right)^2+1\right),\\
\\
Q_\alpha=\left(\left(\frac{\alpha^2-1}{4\alpha}\right)^2-1,\left(\frac{\alpha^2-1}{4\alpha}\right)^2+1\right),\\
\\
R_\alpha=\left(1,\frac{\alpha^2+1}{2\alpha}\right).
\end{array}
\end{equation*}
 When we specialize to $\alpha=2$, we obtain a set of points $S=\{P_2,Q_2,R_2\}=\left\{(0,\frac{73}{64}),(\frac{-55}{64},\frac{73}{64}),(1,\frac{5}{4})\right\}$,
 on
 \begin{equation}
 E_2: y^2=x^3-(\frac{55}{64})^2x+(\frac{73}{64})^2.
 \end{equation}
Using SAGE \cite{8}, one can easily check that \emph{associated height matrix} of $S$ has non-zero determinant $\approx73.3583597733868\neq0$, showing that these three points are independent and so $\text{rank}(E_2)\geq3$. ( Actually the rank is $4$.)  \emph{Specialization} result of Silverman \cite{10} implies that for all but finitely many rational numbers, the rank of $E_\alpha$ is at least $3$. For the values $\alpha=4,10$, and $\alpha=8,11$, the rank of $E_\alpha$ is equal to $5$ and $6$, respectively.
 \end{proof}
\section{ The EC $y^2=x^3-a^2x+b^2$}
\noindent  We study the elliptic curve
\begin{equation}\label{y}
E_t:y^2=x^3-(t^2-1)^2x+(2t)^2,
\end{equation} where $t\in\mathbb{Q}$.
We construct a subfamily with rank at least $3$.
\begin{thm}
There are infinitely many elliptic curves of the form \eqref{y} with rank $\geq3$.
\end{thm}
\begin{proof}
Clearly we have two points
\begin{equation}
P_1=(0,2t),\quad P_2=(t^2-1,2t).
\end{equation}

\noindent Letting $-(t^2-1)^2x+(2t)^2=0$, in \eqref{y}, yields $x=(\frac{2t}{t^2-1})^2$, and $y=(\frac{2t}{t^2-1})^3$. Then the third point is $P_3=((\frac{2t}{t^2-1})^2,(\frac{2t}{t^2-1})^3)=(\frac{b^2}{a^2},\frac{b^3}{a^3})$. By lemma $2$, if $(a,b,c)$ is a PPT, then this point is of infinite order, because $(a,b)=1$, $a\neq1$, and 
 the numbers $\frac{b^2}{a^2}$, and $\frac{b^3}{a^3}$ are not integers. 

\noindent  If we let $t=4T^3$, and $x^3+(2t)^2=0$, then we get $x=-4T^2$, and $y=2T(16T^6-1)$.
 Then the point $P_4=(-4T^2,2T(16T^6-1))$ is on the elliptic curve \eqref{y}.

\noindent When we specialize to $T=1$, we obtain a set of points $A=\{P_1,P_2,P_3,P_4\}=\left\{(0,8),(15,8),((\frac{8}{15})^2,(\frac{8}{15})^3),(-4,30)\right\}$,
 lying on
 \begin{equation}
 E_2: y^2=x^3-(15^2)x+(8^2).
 \end{equation}
  Using SAGE ,  one can easily check that \emph{associated height matrix} of the points $\{P_1, P_2, P_3\}$ or $\{P_2,P_3,P_4\}$ has non-zero determinant $\approx 7.34210213314542\neq0$ , showing that these three points are independent and so the rank of the elliptic curve \eqref{y} is at least $3$, (Actually the rank is $4$.). \emph{Specialization} result of Silverman implies that for all but finitely many rational numbers, the rank of $E_T$ is at least $3$. For the value $T=2$, the rank $E_T$ is equal to $5$.
\end{proof} 

\section{ The EC $y^2=x^3-b^2x+a^2$}
\noindent  We consider the elliptic curve
\begin{equation}\label{f}
E_t:y^2=x^3-(2t)^2x+(t^2-1)^2,
\end{equation} 
where $t\in\mathbb{Q}$, and construct a subfamily with rank at least $3$.
\begin{thm}
There are infinitely many elliptic curves of the form \eqref{y} with rank $\geq3$.
\end{thm}
\begin{proof}
Clearly we have two points
\begin{equation}
P_1=(0,t^2-1),\quad P_2=(2t,t^2-1).
\end{equation}

\noindent Letting $-(2t)^2x+(t^2-1)^2=0$, in \eqref{f}, yields $x=(\frac{t^2-1}{2t})^2$, and $y=(\frac{t^2-1}{2t})^3$. Then the third point is $P_3=((\frac{t^2-1}{2t})^2,(\frac{t^2-1}{2t})^3)=(\frac{a^2}{b^2},\frac{a^3}{b^3})$. Again by lemma $2$, if $(a,b,c)$ is a PPT, this point is of infinite order, because $(a,b)=1$, $b\neq1$, and 
 the numbers $\frac{a^2}{b^2}$, and $\frac{a^3}{b^3}$ are not integers.
 Now we impose a point on \eqref{f} with $x$-coordinate equal to $-1$. Then we have $y^2=t^2(t^2+2)$. It implies that
$t^2+2$
to be a square, say $= \alpha^2$. Hence
$t=\frac{1}{m}-\frac{m}{2}$, and $\alpha=\frac{m}{2}+\frac{1}{m}$.
with $m\in\mathbb{Q}$.
Then the point $P_4=(-1,\frac{1}{m^2}-\frac{m^2}{4})$ is on the elliptic curve \eqref{f}.

\noindent When we specialize to $m=10$($t=\frac{-49}{10}$), we obtain a set of points $A=\{P_1,P_2,P_3,P_4\}=\left\{(0,\frac{2301}{100}),(\frac{-49}{5},\frac{2301}{100}),((\frac{2301}{980})^2,-(\frac{2301}{980})^3),(-1,\frac{-2499}{100})\right\}$,
 lying on
 \begin{equation}\label{g}
 E_2: y^2=x^3-(\frac{49}{5})^2x+(\frac{2376}{25})^2.
 \end{equation}
  Using SAGE,  one can easily check that \emph{associated height matrix} of the points $\{P_1, P_3, P_4\}$ and $\{P_2,P_3,P_4\}$ has non-zero determinant $\approx 421.718713884796$, and $105.429678471199$, respectively. This shows that these three points are independent and so the rank of the elliptic curve \eqref{g} is at least $3$, (Actually the rank is $5$.). \emph{Specialization} result of Silverman implies that for all but finitely many rational numbers, the rank of $E_m$ is at least $3$. For the values $m=3, 5, 6, 7, 8, 10, 11, 13$, and $m=12, 14$, the rank of $E_m$ is equal to $5$, and $6$, respectively.
 \end{proof}
 \section{ The EC $y^2=x^3-a^2x+b^2$}
 
 \noindent We consider the elliptic curve
\begin{equation}\label{y}
E_t:y^2=x^3-(t^2-1)^2x+(2t)^2,
\end{equation} where $t\in\mathbb{Q}$,
 and construct a subfamily with rank at least $3$.
\begin{thm}
There are infinitely many elliptic curves of the form \eqref{y} with rank $\geq3$.
\end{thm}
\begin{proof}
Clearly we have two points
\begin{equation}
P_1=(0,2t),\quad P_2=(t^2-1,2t).
\end{equation}

\noindent Letting $-(t^2-1)^2x+(2t)^2=0$, in \eqref{y}, yields $x=(\frac{2t}{t^2-1})^2$, and $y=(\frac{2t}{t^2-1})^3$. Then the third point is $P_3=((\frac{2t}{t^2-1})^2,(\frac{2t}{t^2-1})^3)=(\frac{b^2}{a^2},\frac{b^3}{a^3})$.  This point is of infinite order, because in a PPT we have $(a,b)=1$ and $a\neq1$, i.e., the numbers $\frac{b^2}{a^2}$ and $\frac{b^3}{a^3}$ are not integers, then the rank of the above elliptic curve is positive.
 If we let $t=4T^3$, and $x^3+(2t)^2=0$, then we get $x=-4T^2$, and $y=2T(16T^6-1)$.
 Then the point $P_4=(-4T^2,2T(16T^6-1))$ is on the elliptic curve \eqref{y}.

\noindent When we specialize to $T=1$, we obtain a set of points $A=\{P_1,P_2,P_3,P_4\}=\left\{(0,8),(15,8),((\frac{8}{15})^2,(\frac{8}{15})^3),(-4,30)\right\}$,
 lying on
 \begin{equation}
 E_2: y^2=x^3-(15^2)x+(8^2).
 \end{equation}
  Using SAGE,  one can easily check that \emph{associated height matrix} of the points $\{P_1, P_2, P_3\}$ or $\{P_2,P_3,P_4\}$ has non-zero determinant $\approx 7.34210213314542\neq0$ , showing that these three points are independent and so the rank of the elliptic curve \eqref{y} is at least $3$, (Actually the rank is $4$.). \emph{Specialization} result of Silverman  implies that for all but finitely many rational numbers, the rank of $E_T$ is at least $3$. For the value $T=2$, the rank $E_T$ is equal to $5$.
 \end{proof}

\section{ The EC $y^2=x^3-c^2x+b^2$}
\noindent  We study the elliptic curve
\begin{equation}\label{h}
E_t:y^2=x^3-(t^2+1)^2x+(2t)^2,
\end{equation} where $t\in\mathbb{Q}$.
We construct a subfamily with rank at least $3$.
\begin{thm}
There are infinitely many elliptic curves of the form \eqref{h} with rank $\geq3$.
\end{thm}
\begin{proof}
Clearly we have two points
\begin{equation}
P_1=(0,2t),\quad P_2=(t^2+1,2t).
\end{equation}

\noindent Letting $-(t^2+1)^2x+(2t)^2=0$, in \eqref{h}, yields $x=(\frac{2t}{t^2+1})^2$, and $y=(\frac{2t}{t^2+1})^3$. Then the third point is $P_3=((\frac{2t}{t^2+1})^2,(\frac{2t}{t^2+1})^3)=(\frac{b^2}{c^2},\frac{b^3}{c^3})$. This point is of infinite order, because in a PPT, we have $(b,c)=1$ and $c\neq1$, i.e., the numbers $\frac{b^2}{c^2}$ and $\frac{b^3}{c^3}$ are not integers, then the rank of the above elliptic curve is positive.
 Now we impose a point on \eqref{f} with $x$-coordinate equal to $1$. Then we have $y^2=t^2(-t^2+2)$. It implies that
$-t^2+2$
to be a square, say $= \alpha^2$. Hence we can get 
$t=\frac{u^2-2u-1}{u^2+1}$, and $\alpha=\frac{-u^2-2u+1}{u^2+1}$.
with $u\in\mathbb{Q}$.
Then the point $P_4=(1,\frac{(-u^2-2u+1)(u^2-2u-1)}{(u^2+1)^2})$ is on the elliptic curve \eqref{h}.

\noindent When we specialize to $u=2$($t=\frac{-1}{5}$), we obtain a set of points $A=\{P_1,P_2,P_3,P_4\}=\left\{(0,\frac{-2}{5}),(\frac{26}{25},\frac{-2}{5}),((\frac{5}{13})^2,-(\frac{5}{13})^3),(1,\frac{7}{25})\right\}$,
 lying on
 \begin{equation}\label{p}
 E_\frac{-1}{5}: y^2=x^3-(\frac{26}{25})^2x+(\frac{2}{5})^2.
 \end{equation}
  Using SAGE,  one can easily check that \emph{associated height matrix} of the points $\{P_1, P_2, P_4\}$ or $\{P_2,P_3,P_4\}$ has non-zero determinant $\approx 16.9957115044387$. (The determinant of points $\{P_1, P_3, P_4\}$ is non-zero, too.)  This shows that these two points ( in each set) are independent and so the rank of the elliptic curve \eqref{p} is at least $3$, (Actually the rank is $5$.). \emph{Specialization} result of Silverman implies that for all but finitely many rational numbers, the rank of $E_u$ is at least $3$.
 \end{proof} 
 \section{ The EC $y^2=x^3-b^2x+c^2$}
\noindent  We study the elliptic curve
\begin{equation}\label{h}
E_t:y^2=x^3-(2t)^2x+(t^2+1)^2,
\end{equation} where $t\in\mathbb{Q}$.
We construct a subfamily with rank at least $2$.
\begin{thm}
There are infinitely many elliptic curves of the form \eqref{h} with rank $\geq2$.
\end{thm}
\begin{proof}
Clearly we have two points
\begin{equation}
P_1=(0,t^2+1),\quad P_2=(2t,t^2+1).
\end{equation}

\noindent Letting $-(2t)^2x+(t^2+1)^2=0$, in \eqref{h}, yields $x=(\frac{t^2+1}{2t})^2$, and $y=(\frac{t^2+1}{2t})^3$. Then the third point is $P_3=((\frac{t^2+1}{2t})^2,(\frac{t^2+1}{2t})^3)=(\frac{c^2}{b^2},\frac{c^3}{b^3})$. Note that this point is of infinite order, because in a PPT, we have $(b,c)=1$ and $b\neq1$, i.e., the numbers $\frac{c^2}{b^2}$ and $\frac{c^3}{b^3}$ are not integers, then the rank of the aforementioned elliptic curve is positive.
 If we impose a point on \eqref{f} with $x$-coordinate equal to $2$. Then we get the point $P_4=(2,t^2-3)$. 

\noindent When we specialize to $t=\frac{7}{29}$, we obtain a set of points $A=\{P_1,P_2,P_3,P_4\}=\left\{(0,\frac{890}{841}),(\frac{14}{29},\frac{890}{841}),((\frac{445}{203})^2,-(\frac{445}{203})^3),(2,\frac{2474}{841})\right\}$,
 lying on
 \begin{equation}\label{q}
 E_\frac{7}{29}: y^2=x^3-(\frac{14}{29})^2x+(\frac{890}{841})^2.
 \end{equation}
  Using SAGE,  one can easily check that \emph{associated height matrix} of the points $\{P_3, P_4\}$ and $\{P_1,P_3\}$ have non-zero determinants $\approx 13.2385415745155$, and $52.9541662980621$, respectively.  This shows that these two points (in each set) are independent and so the rank of the elliptic curve \eqref{q} is at least $2$, (Actually the rank is $4$.). \emph{Specialization} result of Silverman implies that for all but finitely many rational numbers, the rank of $E_t$ is at least $2$. 
 \end{proof}

  \section{ The EC $y^2=x(x-a^2)(x+c^2)$}
\begin{thm} Let $(a,b,c)$ be a PPT. Then the rank of the aforementioned elliptic curve is positive.
\end{thm}

\begin{proof}   
\noindent We have  $y^2=x(x-a^2)(x+c^2)=x(x^2+(c^2-a^2)x-a^2c^2)=x(x^2+b^2x-a^2c^2)=x^3+b^2x^2-a^2c^2x$. Then it suffices that we study the elliptic curve
\begin{equation}\label{z}
y^2=x^3+b^2x^2-a^2c^2x.
\end{equation}
Note that $D=a^4c^4(b^4+4a^2c^2)\neq0$.
Now if in \eqref{z}, we take $b^2x^2-a^2c^2x=0$, then we get $x=\frac{a^2c^2}{b^2}$, and
$y=\frac{a^3c^3}{b^3}$. Therefore the first point on \eqref{z} is $P_1=(\frac{a^2c^2}{b^2},\frac{a^3c^3}{b^3})$. Note that  the order of this point is infinite, because in a PPT, the number $ac$ is not divisable by $b$, and, the numbers $\frac{a^2c^2}{b^2}$ and $\frac{a^3c^3}{b^3}$ are not integers. (Otherwise if $p$ is a prime number that divides $a$, then $p$ must divide one of  $b$, $c$. Now in view of the relation $a^2+b^2=c^2$, p divides $a$, $b$, and $c$, that is not correct, because $(a,b,c)$ is a PPT: $(a,b,c)=1$.) Then the rank of the elliptic curve \eqref{z} is always positive.
 if we let $x^3+b^2x^2=0$, then we get $x=-b^2$, and $y=abc$. Then the second point on \eqref{z} is the point $P_2=(-b^2,abc)$.
Letting $x^3-a^2c^2x=0$, yields the third and fourth points $P_{3,4}=(\pm ac,abc)$.
\end{proof}
\begin{rem}
Note that if in a PPT, $(a,b,c)$, $b$ is odd, then we may by another method prove that the rank of the aforementioned elliptic curve is positive. We prove that in the point $P_2=(-b^2,abc)$, the number $abc$ does not divide $D$, otherwise $abc$ must divide $4a^6c^6$. Then $b$ divides $a^6c^6$, because $b$ is odd. This is not correct, because $(a,b,c)$ is a PPT. Then the point $P_2$, is of infinite order. Now the result follows.
\end{rem} 
  
 \section{ The EC $y^2=x(x-b^2)(x+c^2)$}
\begin{thm} Let $(a,b,c)$ be a PPT. Then the rank of the above elliptic curve is positive. 
\end{thm}
\begin{proof} 
We have  $y^2=x(x-b^2)(x+c^2)=x(x^2+(c^2-b^2)x-b^2c^2)=x(x^2+a^2x-b^2c^2)=x^3+a^2x^2-b^2c^2x$. Then it suffices that we study the elliptic curve
\begin{equation}\label{m}
y^2=x^3+a^2x^2-b^2c^2x.
\end{equation}
Note that $D=b^4c^4(a^4+4b^2c^2)\neq0$.
If in \eqref{m}, we take $a^2x^2-b^2c^2x=0$, then we get $x=\frac{b^2c^2}{a^2}$, and
$y=\frac{b^3c^3}{a^3}$. Then the first point on \eqref{m} is $P_1=(\frac{b^2c^2}{a^2},\frac{b^3c^3}{a^3})$. Note that  the order of this point is infinite, because in a PPT   the number $bc$ is not divisable by $a$, and, the numbers $\frac{b^2c^2}{a^2}$ and $\frac{b^3c^3}{a^3}$ are not integers, this can be similarly proven. Then we conclude that the rank of the elliptic curve \eqref{m} is always positive.
By letting $x^3+a^2x^2=0$, we get $x=-a^2$, and $y=abc$. Then the second point on \eqref{m} is the point $P_2=(-a^2,abc)$. 
Letting $x^3-b^2c^2x=0$, yields the third and fourth points $P_{3,4}=(\pm bc,abc)$.
 \end{proof}   
  
\begin{rem}
Note that if in a PPT, $(a,b,c)$, $a$ is odd, then we may by another method prove that the rank of the aforementioned elliptic curve is positive. We prove that in the point $P_2=(-a^2,abc)$, the number $abc$ does not divide $D$, otherwise $abc$ must divide $4b^6c^6$. Then $a$ divides $b^6c^6$, because $a$ is odd. This is not correct, because $(a,b,c)$ is a PPT. Then the point $P_2$, is of infinite order. Now the result follows.
\end{rem}


\end{document}